\documentclass{article}

\usepackage{amssymb,bm}
\usepackage[english]{babel}
\usepackage[utf8x]{inputenc}

\usepackage{amsmath}
\usepackage{amsfonts}
\usepackage{graphicx}
\usepackage[colorinlistoftodos]{todonotes}
\usepackage[colorlinks=true, allcolors=blue]{hyperref}

\usepackage{amsthm}
\theoremstyle{plain}
\newtheorem{theorem}{Theorem}[section]
\newtheorem{lemma}[theorem]{Lemma}
\newtheorem{prop}[theorem]{Proposition}

\theoremstyle{definition}

\newtheorem{eg}[theorem]{Example}

\newtheorem{rmk}[theorem]{Remark}

\newcommand{\Q}{\mathbb{Q}}
\newcommand{\Z}{\mathbb{Z}}

\newcommand{\R}{\mathbb{R}}
\newcommand{\C}{\mathbb{C}}
\newcommand{\F}{\mathbb{F}}
\newcommand{\proj}{\mathbb{P}}

\newcommand{\OK}{\mathcal{O}}
\DeclareMathOperator{\car}{char}
\DeclareMathOperator{\tr}{tr}
\DeclareMathOperator{\Jac}{Jac}
\DeclareMathOperator{\Gal}{Gal}
\DeclareMathOperator{\Aut}{Aut}
\DeclareMathOperator{\Frob}{Frob}
\DeclareMathOperator{\Ind}{Ind}
\DeclareMathOperator{\sgn}{sgn}
\DeclareMathOperator{\GL}{GL}
\DeclareMathOperator{\red}{red}
\DeclareMathOperator{\lift}{lift}
\DeclareMathOperator{\Dic}{Dic}
\DeclareMathOperator{\id}{id}

\newcommand{\lagrange}[2]{\left( \dfrac{#1}{#2}\right)}

\usepackage{tikz}
\usetikzlibrary{cd}

\begin{document}

	\title{Wild Galois representations: a family of hyperelliptic curves with large inertia image}

	\author{NIRVANA COPPOLA\\
		Vrije Universiteit Amsterdam\\
		e-mail\textup{: \texttt{n.coppola@vu.nl}}}
	

	\maketitle

	\begin{abstract}
		In this work we generalise the main result of \cite{1812.05651} to the family of hyperelliptic curves with potentially good reduction over a $p$-adic field which have genus $g=\dfrac{p-1}{2}$ and the largest possible image of inertia under the $\ell$-adic Galois representation associated to its Jacobian. We will prove that this Galois representation factors as the tensor product of an unramified character and an irreducible representation of a finite group, which can be either equal to the inertia image (in which case the representation is easily determined) or a $C_2$-extension of it. In this second case, there are two suitable representations and we will describe the Galois action explicitly in order to determine the correct one.
	\end{abstract}
	
	\section{Introduction}\label{intro}
	A \emph{hyperelliptic curve} over a field $K$ is a smooth, projective, geometrically integral algebraic curve $X$ of genus $g \geq 1$ that has the structure of a degree $2$-cover of $\proj^1$. Elliptic curves are special cases of hyperelliptic curves, with $g=1$ (and a rational point). Similarly as with elliptic curves, when $\car (K) \neq 2$, any hyperelliptic curve over $K$ can be defined by an affine Weierstrass equation of the form
	\[
	X : y^2 = f(x),
	\]
	where $f(x) \in K[x]$ is a square-free polynomial. By this, we mean that the function field of $X$ is isomorphic to $K(x)[y]/(y^2-f(x))$, as a quadratic extension of $K(x)$. It is also well known that $\deg(f)$ is equal to $2g+1$ or $2g+2$.
	
	One important difference between elliptic curves and higher genus hyperelliptic curves is that the set of points on the latter does not have a group structure. However, it is possible to associate an abelian variety to any curve $X$, namely the \emph{Jacobian variety} $\Jac(X)$, and study the group structure on it. For the definition of the Jacobian of a curve, see e.g. \cite[\S 1]{milne7}. If $\overline K$ is a fixed separable closure of $K$, we denote by $\Jac(X)(\overline K)$ the set of points defined over $\overline K$ and lying on $\Jac(X)$. In particular we can define, for a prime $\ell$, the $\ell$-adic Tate module, which is
	\[
	T_\ell \Jac(X) = \varprojlim_n \Jac(X)[\ell^n],
	\]
	where by $\Jac(X)[m]$ we denote the subgroup of $m$-torsion points of $\Jac(X)(\overline K)$. It can be proved that, for $\ell$ different from the characteristic of $K$, this is a free $\Z_\ell$-module of rank $2g$ (see \cite[\S 15 Theorem 15.1]{milne5}).
	
	Let $G_K=\Gal (\overline K/K)$ be the absolute Galois group of $K$. Then we have a linear action on the points of $\Jac(X)$, and an induced action on the Tate modules, thus we can define, for any prime $\ell$ (different from $\car (K)$) a Galois representation
	\[
	\rho_\ell : G_K \rightarrow \Aut(T_\ell \Jac(X)),
	\]
	which is a $2g$-dimensional representation. After taking the tensor product with $\overline \Q_\ell$, and fixing a basis for $T_\ell \Jac(X) \otimes \overline{\Q}_\ell$, we can and will consider $\Aut(T_\ell \Jac (X))$ as a subgroup of $\GL_{2g} (\overline{\Q}_\ell)$. 
	
	From now on, we assume that $K$ is a non-archimedean local field of characteristic $0$, i.e. a finite extension of $\Q_p$ for some prime $p$; we also assume that $p \neq \ell$. We denote by $v_K$ the valuation on $K$, by $\OK_K$ the ring of integers, by $\pi_K$ a uniformiser, by $k$ the residue field, with algebraic closure $\overline k$, and by $K^{nr}$ the maximal unramified extension of $K$ contained in $\overline K$. Then the Galois group $G_K$ has a normal subgroup with pro-cyclic quotient, namely the \emph{inertia subgroup}:
	\[
	I_K = \{ \sigma \in G_K | \sigma (x) = x \ \forall x \in \overline k\}.
	\]
	
	In fact, $I_K=\Gal(\overline K / K^{nr})$ and the quotient is isomorphic to $\Gal(\overline k/k)$, thus it is generated by any element of $G_K$ that acts as Frobenius on the residue field, i.e. as $x \mapsto x^{|k|}$. We call any such generator a Frobenius element of $K$, and we denote it by $\Frob_K$. In Section \ref{frobenius} we fix a precise choice of $\Frob_K$. Therefore, in order to describe the Galois representation $\rho_\ell$, we need to compute the image of the inertia elements and the image of Frobenius, either by fixing a basis for $\Aut(T_\ell \Jac(X) )$ and computing the matrices representing these elements, or, as in \cite[Lemma 3]{c9f01b439a3c48dfac7d37fb5b576af4}, by expressing $\rho_\ell$ as a sum of irreducible representations, each equal to the tensor product of an unramified character and a representation of some finite group.
	
	In this work, we shall generalise the main result of \cite{1812.05651} to higher genus curves. Namely, we will consider hyperelliptic curves with bad, potentially good reduction at the largest wild prime. Such curves are defined by an equation $y^2=f(x)$ where $f$ is square-free of odd degree equal to the residue characteristic of $K$, see \cite[\S 2 Corollary 2(a)]{10.2307/1970722}, and we will focus on the case where the image of inertia under $\rho_\ell$ is the largest possible.
	
	In Section \ref{statement} we give the statement of the main result of this work. In Section \ref{inertia} we describe explicitly the action of inertia in our setting, using \cite[Theorem 10.3]{M2D2} and \cite[\S 8.2 Proposition 25]{Serre}. In Section \ref{frobenius} we use a good model for the family of curves we are interested in, to compute the eigenvalues of Frobenius. Finally in Section \ref{proof} we give details about the proof, in particular in the case where the inertia degree of $K/\Q_p$ is odd, when the work of Sections \ref{inertia} and \ref{frobenius} is not sufficient to describe the full Galois representation.
	
	\section{Statement of the main results}\label{statement}
	
	Let $p$ be an odd prime and let $K$ be a finite extension of $\Q_p$. Let $X$ be a hyperelliptic curve over $K$ defined by an equation of the form:
	\[
	X : y^2 = f(x) \qquad \text{with } f \in K[x] \text{ monic, square-free, of degree }p. \quad (*)
	\]
	
	Recall that the genus $g$ of the curve satisfies $p=2g+1$. Suppose that $X$ has potentially good reduction over $K$, i.e. there exists a finite extension $F/K$ such that the base change $X \times_K F$ of the curve $X$ to $F$ has good reduction. Then by the Criterion of N\'eron-Ogg-Shafarevich (see \cite[\S 2 Theorem 2(ii)]{10.2307/1970722}), the Galois representation $\rho_\ell$ restricted to inertia factors through a finite quotient. We assume that this quotient has the largest possible size. The first result characterises the hyperelliptic curves that satisfy this assumption. Let $\alpha_1,\dots,\alpha_p \in \overline K$ be the roots of $f$, $\Delta$ be the discriminant of $f$, $G=\Gal(K(\{\sqrt{\alpha_i-\alpha_j}\}_{i \neq j})/K)$, and let $I_G$ be the inertia subgroup of $G$.
	
	\begin{prop}\label{PROP:characterisation}
		Let $X$ be a hyperelliptic curve over a $p$-adic field $K$ defined by an equation $(*)$ with potentially good reduction, and let $\rho_\ell$ be the $\ell$-adic Galois representation associated to it. Then $\rho_\ell(I_K) \cong I_G$. Moreover, $|\rho_\ell(I_K)|$ is maximal and equal to $2p(p-1)$ if and only if 
		\begin{enumerate}
			\item the Galois group of the splitting field of $f$ over $K^{nr}$ is isomorphic to the Frobenius group $C_p \rtimes C_{p-1}$, and
			\item $v_K(\Delta)$ is odd.
		\end{enumerate} 
		
		The structure of the group $I_G$ when these two conditions hold is that of the semidirect product $C_p \rtimes C_{2(p-1)}$ of $C_p$ and $C_{2(p-1)}$ which has a degree $2$ quotient isomorphic to the Frobenius group $C_p \rtimes C_{p-1}$.
	\end{prop}
	
	The first condition in this proposition is expensive to check computationally, however the following result gives two conditions that imply those above and are easier to verify. Throughout the rest of the paper we will assume for simplicity that these two new conditions hold, however they can be replaced by the general ones, in fact, the main result of this paper (Theorem \ref{maintheorem}) holds whenever the image of inertia is maximal, i.e. whenever assumptions (i) and (ii) of Proposition \ref{PROP:characterisation} are satisfied.
	
	\begin{prop}\label{assumptions}
		The conditions in Proposition \ref{PROP:characterisation} are implied by the following two:
		\begin{enumerate}
			\item $f$ is irreducible over $K$;
			\item $(v_K(\Delta),p-1)=1$.
		\end{enumerate}
	\end{prop}

	For the proof of these statements see Section \ref{inertia}.
	
	Let $F=K(\alpha_1,\dots,\alpha_p,\sqrt{\alpha_2-\alpha_1})$. We will prove that if the conditions in Proposition \ref{assumptions} hold, $F/K$ is totally ramified and $X$ acquires good reduction over $F$. Moreover if the inertia degree $f_{K/\Q_p}$ of $K$ over $\Q_p$ is even, then $F=K(\{\sqrt{\alpha_i-\alpha_j}\}_{i \neq j})$ and so $G=I_G$, otherwise $G$ is isomorphic to a semidirect product of the form $I_G \rtimes C_2$. Let us now fix a numbering on the roots and a $p$-cycle $\sigma \in I_G$ (which exists since $p \mid |I_G|$), such that:
	\[
	\sigma: \alpha_1 \mapsto \alpha_2 \mapsto \dots \alpha_p \mapsto \alpha_1
	\]
	and $\sigma(\sqrt{\alpha_i-\alpha_j})=\sqrt{\sigma(\alpha_i)-\sigma(\alpha_j)} $ for all $i,j$. Moreover, for odd $f_{K/\Q_p}$, let $\phi$ be the non-trivial element of $G$ that fixes the field $F$. For each prime $\ell \neq p$ we fix an embedding $\overline \Q_\ell \rightarrow \C$. In particular we identify $\sqrt{p}$ with the positive real square root of $p$, and $\sqrt{-p}$ with the complex number $i \sqrt{p}$.
	
	Our main result is the following.
	\begin{theorem} \label{maintheorem}
		Let $X/K$ be a hyperelliptic curve over a $p$-adic field $K$ defined by an equation $(*)$, with potentially good reduction. Let $f_{K/\Q_p}$ be the inertia degree of $K/\Q_p$. Let $\rho_\ell$ be the $\ell$-adic Galois representation attached to $\Jac(X)$, for $\ell \neq p$. Suppose that $f$ is irreducible over $K$ and that the valuation of the discriminant of $f$ is coprime to $p-1$. Let $G,I_G,\sigma,\phi$ be as above.
		
		Then $\rho_\ell$ is irreducible and factors as $\rho_\ell = \chi \otimes \psi$, where:
		\begin{eqnarray*}
			\chi: &G_K \rightarrow \overline{\Q}_\ell^\times\\
			&I_K \mapsto 1\\
			&\Frob_K \mapsto \sqrt{\lagrange{-1}{p} p}^{f_{K/\Q_p}}
		\end{eqnarray*}
		and:
		\begin{enumerate}
			\item if $f_{K/\Q_p}$ is even, $\psi$ is the unique irreducible faithful representation of $G=I_G$ of dimension $p-1$;
			\item if $f_{K/\Q_p}$ is odd, $\psi$ is the unique irreducible faithful representation of $G \cong I_G \rtimes C_2$ of dimension $p-1$ such that $\tr (\psi(\sigma \phi))=-\sqrt{\lagrange{-1}{p}p}$.
		\end{enumerate}
	\end{theorem}
	
	The strategy of the proof is as follows: first we determine the Galois representation restricted to inertia, then we find a model of $X \times_K F$ reducing to $y^2=x^p-x$ over the residue field, and use this to determine the action of $\rho_\ell(\Frob_K)$. If the inertia degree $f_{K/\Q_p}$ is even, then this information is enough to determine the Galois representation $\rho_\ell$, otherwise there are two representations that, when restricted to inertia, give the same result, and the two only differ by the trace of the elements that are products of Frobenius with a wild inertia automorphism. We will compute explicitly the trace of one such element, namely $\sigma \Frob_K$ where $\sigma$ is defined above, using again the good model $y^2=x^p-x$, to conclude.

	\begin{eg}\label{example}
		Let $X/\Q_p$ be the curve defined by $y^2 = f(x)= x^p - p$. Then $f$ is irreducible over $\Q_p$ and $v_K(\Delta)= 2p-1$ is relatively prime to $p-1$. Therefore, $X$ satisfies the assumptions of Theorem \ref{maintheorem}.
	\end{eg}
	
	\section{The inertia action} \label{inertia}
	
	In this section we prove Propositions \ref{PROP:characterisation}, \ref{assumptions} and we use Proposition 25 in \cite[\S 8.2]{Serre} to determine the restriction to inertia of $\rho_\ell$.
	
	\subsection{Proof of Propositions \ref{PROP:characterisation} and \ref{assumptions}}\label{sec:proofprops}
	\begin{rmk}[Cluster picture for the curve]\label{cluster}
		
		Recall that for a hyperelliptic curve of the form $y^2=f(x)$, a cluster is a subset of the set of all roots of $f$ in $\overline K$ with the property that the difference of any two different elements of it has valuation $\geq \delta$, for some $\delta \in \R$. For more detailed definitions, see \cite[\S 1]{M2D2}.
		
		Suppose that $X$ is a hyperelliptic curve defined by an equation $(*)$, and that $f(x)$ has roots $\alpha_1,\dots,\alpha_p \in \overline K$; note in particular that, if $g$ is the genus of the curve, then $p=2g+1$. By \cite[Theorem 10.3]{M2D2}, we have that $X$ has potentially good reduction if and only if the cluster picture consists of a unique cluster $R$ of size $p$ containing all the roots. In this case, $\Jac(X)$ also has potentially good reduction.
	\end{rmk}
	
	By the Criterion of N\'eron-Ogg-Shafarevich (see \cite[\S 2 Theorem 2(ii)]{10.2307/1970722}), if $\ell \neq p$ the Galois representation $\rho_\ell$ on $T_\ell \Jac(X)$, restricted to inertia, has finite image, independent of $\ell$. Moreover by Corollary 3 in the same paper, this image is isomorphic to $\Gal(K^{nr}(\Jac(X)[m])/K^{nr})$ for any $m \geq 3$ coprime to $p$, and $K^{nr}(\Jac(X)[m])$ is the minimal extension of $K^{nr}$ over which $\Jac(X)$ acquires good reduction. We can fix $m=4$; then by \cite[Theorem 1.1]{Yelton}, we have that
	\[
	K^{nr}(\Jac(X)[4]) = K^{nr} (\{ \sqrt{\alpha_i - \alpha_j }\}_{i,j \in \{1,\dots,p\}}).
	\]
	
	We denote by $C_p \rtimes C_{2(p-1)}$ the semidirect product of $C_p$ and $C_{2(p-1)}$ which has the following presentation:
	\[
	\langle \sigma, \tau | \sigma^p=\tau^{2(p-1)}=1, \tau \sigma \tau^{-1}=\sigma^b \rangle,
	\]
	for some $b$ coprime to $p$. The exact value of $b$ will not be relevant for the rest of the paper. This group has a degree $2$ quotient isomorphic to the Frobenius group $C_p \rtimes C_{p-1}$, i.e. the group of affine transformations on the finite field with $p$ elements, given by the non-abelian semidirect product $C_p \rtimes \Aut(C_p)$. The extra $C_2$ contained in $C_p \rtimes C_{2(p-1)}$ is generated by $\tau^{p-1}$. We will recall and use this notation in \S \ref{sec:irrepsIG}.

	\begin{lemma}\label{maxinertia}
		With notations above, we have isomorphisms
		\[
		\rho_\ell(I_K) \cong \Gal(K^{nr}(\{\sqrt{\alpha_i-\alpha_j}\}_{i,j})/K^{nr}) \cong I_G
		\]
		and this group is isomorphic to a subgroup of $C_p \rtimes C_{2(p-1)}$.
		
	\end{lemma}

	\begin{proof}
		Recall that $I_G$ is the inertia subgroup of the Galois group of $K(\{\sqrt{\alpha_i-\alpha_j}\}_{i,j})/K$, so we have that $K(\{\sqrt{\alpha_i-\alpha_j}\}_{i,j})K^{nr}/K^{nr}$ is also Galois with Galois group isomophic to $I_G$.
		
		Let $L=K^{nr}(\{\sqrt{\alpha_i-\alpha_j}\}_{i,j})$ and $L' = K^{nr}(\Jac(X)[2])$, and consider the tower of field extensions $L/L'/K^{nr}$. Notice that $L'$ is the splitting field of $f$ over $K^{nr}$, by \cite[Lemma 2.1]{Corn}.
		
		For all $i$, $\alpha_i \in L'$, therefore $L$ is obtained from $L'$ by adjoining square roots of some elements of $L'$. So the Galois group of $L/L'$ is a direct product of some copies of $C_2$. However, it is a totally ramified extension since $L' \supseteq K^{nr}$, and it is tame since $p$ is odd, therefore it must be cyclic. So $L/L'$ can only be trivial or quadratic.
		
		Now let us consider $L'/K^{nr}$. Since $f$ is irreducible over $K^{nr}$, we have that $\Gal(L'/K^{nr})$ has a cyclic subgroup of order $p$. Therefore $\Gal(L'/K^{nr})$ injects into $S_p$, the group of permutations on $p$ elements, and since $p$ divides $|S_p|$ exactly once, necessarily the $p$-Sylow subgroup of $\Gal(L'/K^{nr})$ is isomorphic to $C_p$. So, the wild inertia subgroup of $\Gal(L'/K^{nr})$ is isomorphic to $C_p$, and the quotient by $C_p$ is the Galois group of the maximal tamely ramified subextension of $L'/K^{nr}$, so it is cyclic. Now the image of it in $S_p$ is contained in the normaliser of $C_p$, that is equal to $C_p \rtimes C_{p-1}$. Therefore $\Gal(L'/K^{nr})$ injects into $C_p \rtimes C_{p-1}$.
		
		Putting all this together, if $f$ is irreducible over $K^{nr}$ then $\Gal(L/K^{nr})$ has at most order $2 \cdot p(p-1)$, with wild inertia subgroup of order $p$ and a cyclic quotient of order at most $2(p-1)$, corresponding to the maximal tame subextension. Since $L'/K^{nr}$ is an intermediate subextension with Galois group isomorphic to a subgroup of the Frobenius group $C_p \rtimes C_{p-1}$, this concludes the proof.
	\end{proof}
	
	This lemma shows that, for a curve defined as in $(*)$ acquiring good reduction over a wildly ramified extension, the size of the image of inertia under $\rho_\ell$ is at most $2p(p-1)$. In fact equality can be achieved, and in this case we say that the curve has maximal inertia image. We now complete the proof of Proposition \ref{PROP:characterisation}.
	
	\begin{proof}[Proof of Proposition \ref{PROP:characterisation}.]
		In Lemma \ref{maxinertia}, we proved that $|\rho_\ell(I_K)|$ divides $2p(p-1)$, and that (with the same notation used in the proof) $[L:L'] \leq 2$ and $[L':K^{nr}] \leq p(p-1)$. Clearly, the second inequality is an equality precisely when the Galois group of the splitting field of $f$ over $K^{nr}$ is $C_p \rtimes C_{p-1}$. Moreover, we have:
		\[
		L=L'(\sqrt{\alpha_2-\alpha_1})=K(\alpha_1,\dots,\alpha_p,\sqrt{\alpha_2-\alpha_1}),
		\]
		in fact at most one of the elements $\sqrt{\alpha_i-\alpha_j}$ is sufficient to generate $L$ over $L'$ and by Remark \ref{cluster} any of these elements works as they all have the same valuation. More precisely, the extension $L/L'$ is quadratic if and only if $\alpha_2-\alpha_1$ is not a square in $L'$, or equivalently it has odd valuation. We denote by $v_L$, $v_{L'}$ the normalised valuations on $L$ and $L'$ respectively. Then $v_L(\sqrt{\alpha_2-\alpha_1})=\dfrac{1}{2}[L:L']v_{L'}(\alpha_2-\alpha_1)$ and since by definition $\Delta = \prod_{i > j} (\alpha_i - \alpha_j)^2$, we have:
		\[
		v_L(\Delta) = \binom{p}{2} 2 v_L(\alpha_2-\alpha_1) = \binom{p}{2} 4 v_L(\sqrt{\alpha_2-\alpha_1}) =  2p(p-1) v_L(\sqrt{\alpha_2-\alpha_1});
		\]
		on the other hand since the valuations on $K^{nr}$ and $K$ agree on the elements of $K$ we have $v_L(\Delta)=[L:K^{nr}]v_K(\Delta)$.
		
		Suppose that $|\rho_{\ell}(I_K)|=2p(p-1)$, so $[L:K^{nr}]=2p(p-1)$. In particular, $[L:L']=2$ and by the observation above this means $v_{L'}(\alpha_2-\alpha_1)$ is odd. Then simplifying from the equalities above we obtain that $v_K(\Delta)=v_L(\sqrt{\alpha_2-\alpha_1}) = v_{L'}(\alpha_2-\alpha_1)$ is odd and, as we already noted, $\Gal(L'/K^{nr}) \cong C_p \rtimes C_{p-1}$. Conversely suppose that $\Gal(L'/K^{nr}) \cong C_p \rtimes C_{p-1}$ and that $v_K(\Delta)$ is odd. Then comparing the two expressions for $v_L(\Delta)$ and using that $[L':K^{nr}]=p(p-1)$ we obtain
		\[
		[L:L'] v_K(\Delta) = 2 \cdot \dfrac{1}{2}[L:L']v_{L'}(\alpha_2-\alpha_1), 
		\]
		so $v_{L'}(\alpha_2-\alpha_1)=v_K(\Delta)$ is odd, which implies $[L:L']=2$ and therefore $[L:K^{nr}]=2p(p-1)$.
		%
		%
	\end{proof}
	
	To conclude this subsection, we prove Proposition \ref{assumptions}.
	
	\begin{proof}[Proof of Proposition \ref{assumptions}.]
		Since $p−1$ is even, it follows from assumption (ii) that $v_K(\Delta)$ is odd. So, in order to prove Proposition \ref{assumptions} it suffices to prove that, if $f$ is irreducible over $K$ and $(v_K(\Delta),p-1)=1$ then $\Gal(L'/K^{nr}) \cong C_p \rtimes C_{p-1}$, where $L'$ is as in the proofs of Lemma \ref{maxinertia} and Proposition \ref{PROP:characterisation}.
		
		Let us denote by $M$ the splitting field of $f$ over $K$. First of all, since $f$ is irreducible over $K$, then $\Gal(M/K)$ contains a subgroup $H$ of order $p$. We observe that $H$ is a $p$-Sylow subgroup of $\Gal(M/K)$, since this group has a natural injection into $S_p$, whose $p$-Sylow subgroups have order $p$. Therefore $p$ divides exactly $[M:K]$. Let $U$ be the maximal unramified extension of $K$ contained in $M$. If $p \mid [M:U]$ then $H$ is the first ramification group of $\Gal(M/K)$ (see \cite[Chapter IV \S 1]{GreenbergSerre}). Otherwise, let $\alpha_1$ be any root of $f$ in $M$ and consider $K'=K(\alpha_1)$, which has (prime) degree $p$ over $K$. Since $p \nmid [M:U]$, we have $K \subseteq K' \subseteq U$. Now, $U/K$ is cyclic, therefore every subextension of it is Galois. In particular, $K'/K$ is Galois, thus $K'$ coincides with the splitting field of $f$, hence $K'=U=M$. In both cases, $H$ is normal in $\Gal(M/K)$.

		
		Now consider the element $\sqrt[p-1]{\Delta}$. Using the expression of $\Delta$ in terms of the roots $\alpha_1,\dots,\alpha_p$ of $f$ and the fact that $\alpha_i-\alpha_j$ all have the same valuation, we can prove that $\sqrt[p-1]{\Delta} \in L'$. More precisely, we have
		\[
		\sqrt[p-1]{\Delta} = (\alpha_2-\alpha_1)^p \sqrt[p-1]{u}, \quad \text{where } u=\prod_{i>j} \left(\dfrac{\alpha_i-\alpha_j}{\alpha_2-\alpha_1}\right)^2;
		\]
		note that $u$ is an element of $L'$ with valuation $0$ and so its $(p-1)$-th root gives an unramified hence trivial extension of $L'$. Since $(v_K(\Delta),p-1)=1$, we also have that $[K^{nr}(\sqrt[p-1]{\Delta}):K^{nr}]=p-1$. Therefore $p-1$ divides both $[L':K^{nr}]$ and $[M:K]$. Now the inertia subgroup of $\Gal(M/K)$ is isomorphic to $\Gal(L'/K^{nr})$, and it is normal with cyclic quotient. Therefore it must contain the subgroup $H$ of $\Gal(M/K)$ isomorphic to $C_p$. This proves that $p$ and $p-1$ divide $[L':K^{nr}]$ and as in the proof of Lemma \ref{maxinertia} we conclude that $\Gal(L'/K^{nr}) \cong C_p \rtimes C_{p-1}$.
	\end{proof}
	
	\subsection{The irreducible representations of the group $C_p \rtimes C_{2(p-1)}$}\label{sec:irrepsIG}
	
	We now want to describe the representation induced from $\rho_\ell$ on $I_G$.
	In order to do it, we make a digression on the irreducible representations of the group $C_p \rtimes C_{2(p-1)}$.
	
	Let $\sigma,\tau \in C_p \rtimes C_{2(p-1)}$ be as in the presentation given in \S \ref{sec:proofprops}, and denote by $\nu$ the element $\tau^{p-1}$, that generates the extra $C_2$ contained in $C_p \rtimes C_{2(p-1)}$.
	Note that $\nu$ is the only element of the subgroup $C_{2(p-1)}$ of $C_p \rtimes C_{2(p-1)}$ (except the identity) that commutes with $\sigma$. The group $C_p \rtimes C_{2(p-1)}$ is the semidirect product of two abelian groups, $A= C_p$ and $H=C_{2(p-1)}$, so we are in the setting of \cite[\S 8.2]{Serre}. Consider a set of representatives for the orbits of $H$ in the group of characters of $A$. This set consists of two elements only, namely the trivial representation $\mathbf{1}$ and a non-trivial character $\eta$. Let $H_{\mathbf{1}}$ (resp. $H_\eta$) denote the subgroup of $H$ consisting of the elements that stabilise $\mathbf{1}$ (resp. $\eta$). Then $H_{\mathbf{1}} = H$ and $H_\eta=\langle \nu \rangle \cong C_2$. Now for any irreducible representation $\xi$ of $H_{\bullet}$ we obtain a representation of $G$ given by $\Ind_{A H_{\bullet}}^{C_p\rtimes C_{2(p-1)}} \bullet \otimes \xi$. By \cite[\S 8.2 Proposition 25]{Serre} the representations obtained in this way are exactly all the irreducible representations of $C_p \rtimes C_{2(p-1)}$. This proves the following lemma.
	
	\begin{lemma}
		The group $C_p \rtimes C_{2(p-1)}$ has $2(p-1)$ irreducible representations of dimension $1$, corresponding to the $2(p-1)$ irreducible representations of $H_{\mathbf{1}}=C_{2(p-1)}$, and two representations of dimension $p-1$ corresponding to the two irreducible representations of $H_{\eta} \cong C_2$.
	\end{lemma}
	
	Out of these two $(p-1)$-dimensional irreducible representations, exactly one is faithful. Since $H_\eta \cong C_2$, the representation $\xi$ needed for the construction described above is either the trivial representation of $H_\eta$, or the representation $\sgn$, defined by $\sgn(\nu)=-1$. So we obtain the two representations $\Ind_{C_{2p}}^{C_p\rtimes C_{2(p-1)}} \eta$ and $\Ind_{C_{2p}}^{C_p\rtimes C_{2(p-1)}} \eta \otimes \sgn$.
	
	\begin{enumerate}
		\item The representation $\Ind_{C_{2p}}^{C_p\rtimes C_{2(p-1)}} \eta$ is not faithful. In fact, 
		$$\tr (\Ind_{C_{2p}}^{C_p\rtimes C_{2(p-1)}} \eta) (1)=\tr (\Ind_{C_{2p}}^{C_p\rtimes C_{2(p-1)}} \eta) (\nu)= p-1.$$
		\item The representation $\Ind_{C_{2p}}^{C_p\rtimes C_{2(p-1)}} \eta \otimes \sgn$ is faithful. In fact we have, for $s \in C_p\rtimes C_{2(p-1)}$ and for $t_1,\dots,t_{p-1}$ a set of representatives for $C_p\rtimes C_{2(p-1)}/C_{2p}$:
		\[
		\tr (\Ind_{C_{2p}}^{C_p\rtimes C_{2(p-1)}} \eta \otimes \sgn) (s)= \sum_{i: t_i s t_i^{-1} \in C_{2p}} \eta(t_i s t_i^{-1}) \sgn (s).
		\]
		For the terms occurring in this sum (which are at most $p-1$) we have that $\eta(t_i s t_i^{-1})$ is some root of unity, and it is $1$ if and only if $s=1$. So for $s \neq 1$ we have a sum of at most $p-1$ roots of unity, different from $1$, and therefore $\tr (\Ind_{C_{2p}}^{C_p\rtimes C_{2(p-1)}} \eta \otimes \sgn) (s) \neq p-1$, or equivalently the representation is faithful.
	\end{enumerate}

	Now we can determine the representation $\rho_\ell \big|_{I_K}$. Since it factors through $I_G$, which is isomorphic to $C_p \rtimes C_{2(p-1)}$, we identify it with a representation of the finite group $C_p\rtimes C_{2(p-1)}$, and as such, it is faithful. Furthermore, the discussion above implies the following lemma.
	
	\begin{lemma}\label{rhoirred}
		The restriction to inertia of the representation $\rho_\ell$ is irreducible.
	\end{lemma}
	
	\begin{proof}
		Suppose that $\rho_{\ell} \big|_{I_K}$ is reducible. Then, since it has dimension equal to $2g=p-1$, it is the sum of $p-1$ one-dimensional representations, but in this case the image would be abelian. However, this representation factors through $I_G$ and it is faithful as a $I_G$-representation, so since $I_G$ is non-abelian we have a contradiction. Therefore $\rho_\ell \big|_{I_K}$ must be irreducible.
	\end{proof}
	
	Note that as a consequence of this, the representation $\rho_\ell$ is also irreducible. This proves the following result.
	\begin{prop}\label{I_faithful}
		Let $X$ be a hyperelliptic curve over a $p$-adic field $K$ defined by an equation $(*)$, with potentially good reduction. Let $\rho_\ell$ be the $\ell$-adic Galois representation attached to $\Jac(X)$, for $\ell \neq p$. Suppose that $X$ satisfies the conditions of Proposition \ref{PROP:characterisation}, and let $I_G$ be as in \S \ref{statement}. Then the representation $\rho_\ell$ restricted to inertia factors through $I_G \cong C_p \rtimes C_{2(p-1)}$ and, as a representation of $I_G$, it is the unique irreducible faithful representation of dimension $p-1$.
	\end{prop}
	
	\begin{rmk}\label{m2d2}
		By \cite[Theorem 10.1]{M2D2}, the representation given by the action of $I_K$ on the first \'etale cohomology group $H^1_{\acute{e}t}(X \times_K \overline K,\Q_\ell)$ is given by
		\[
		\rho^* = \gamma \otimes (\Q_\ell [R] \ominus \mathbf{1}),
		\]
		where $\gamma$ is a certain character of order $2(p-1)$. This is immediate from op. cit. since the cluster picture of $X$ only contains the cluster $R$ described in Remark \ref{cluster}.
		
		Now, the representation $\rho_\ell\big|_{I_K}$ is dual to $\rho^*$, by \cite[Theorem 15.1]{corsil}, but since by \cite[Theorem 2 (ii)]{10.2307/1970722} its character has integer values, it is in fact isomorphic to it.
	\end{rmk}
	
	\section{The good model and the action of Frobenius} \label{frobenius}
	
	In this section we show that any hyperelliptic curve satisfying the hypotheses of Theorem \ref{maintheorem} has a good model, i.e. an integral model with good reduction, over the field $F$ defined in Section \ref{statement}, that reduces to
	\[
	y^2 = x^p -x
	\]
	on the residue field. We then prove that the action of Frobenius is diagonalisable, with eigenvalues:
	\begin{enumerate}
		\item all equal to $\left(\lagrange{-1}{p} p\right)^{f_{K/\Q_p}/2}$, if $f_{K/\Q_p}$ is even;
		\item half equal to $\left(\lagrange{-1}{p} p\right)^{f_{K/\Q_p}/2}$ and half equal to $-\left(\lagrange{-1}{p} p\right)^{f_{K/\Q_p}/2}$, if $f_{K/\Q_p}$ is odd.
	\end{enumerate}
	
	In particular we deduce that the full Galois representation $\rho_\ell$ is completely determined by these data when $f_{K/\Q_p}$ is even.
	
	As observed in Section \ref{inertia}, the extension $F/K$ is totally ramified, so the residue fields of $F$ and $K$ are both equal to $k$. Therefore we will identify the action of $\Frob_K$ with that of $\Frob_F$, which is well defined.
	
	\begin{lemma}\label{lem:good_model}
		Under the assumptions of Theorem \ref{maintheorem}, the base change of $X$ to $F$ has an integral model reducing to $y^2=x^p-x$ on $k$.
	\end{lemma}
	
	\begin{proof}
		Over $F$, we can define the following change of variables:
		\[
		\left\lbrace
		\begin{array}{l}
			x \mapsto (\alpha_2-\alpha_1) x + \alpha_1  \\
			y \mapsto \sqrt{\alpha_2 - \alpha_1}^p y.
		\end{array}
		\right.
		\]
		
		Then applying this change of variables to $X \times_K F$ we have the following equation:
		\[
		y^2 = \prod_{1 \leq i \leq p} \left(x - \dfrac{\alpha_i - \alpha_1}{\alpha_2-\alpha_1}\right).
		\]
		
		Note that for each $i \in \{2,\dots,p\}$, we have
		$$ \alpha_i - \alpha_1 = \sum_{j=0}^{i-2} \sigma^j(\alpha_2 - \alpha_1), $$
		where $\sigma$ is the $p$-cycle defined in Section \ref{statement}. Since $\sigma$ is a wild inertia element, $\dfrac{\sigma^j(\alpha_2-\alpha_1)}{\alpha_2-\alpha_1}$ reduces to $1$ on $k$ (see \cite[Lemma 3.3]{1812.05651}), therefore the reduction of $\prod_{1 \leq i \leq p} \left(x - \dfrac{\alpha_i - \alpha_1}{\alpha_2-\alpha_1}\right)$ is $\prod_{1 \leq i \leq p} (x - (i-1)) = x^p-x$.
	\end{proof}
	
	\begin{rmk}
		We know from the Criterion of N\'eron-Ogg-Shafarevich (\cite[\S2 Theorem 2(ii)]{10.2307/1970722}) that $\Jac(X)$ acquires good reduction over $F$; this lemma shows that the curve $X$ itself acquires good reduction over the same extension.
	\end{rmk}
	
	Since $F/K$ is totally ramified of degree $2p(p-1)$, there is an intermediate extension $F'$ such that $F/F'$ is wild of degree $p$ and $F'/K$ is tame of degree $2(p-1)$. Hence there exists some $2(p-1)$-th root of the uniformiser $\pi_K$ of $K$ that generates $F'/K$. Now, since $K$ is a finite extension of $\Q_p$, it contains all the $(p-1)$-th roots of unity; moreover $K$ also contains a primitive $2(p-1)$-th root of unity if and only if the unramified part of the extension $K/\Q_p$ has even degree, i.e. if $f_{K/\Q_p}$ is even. Therefore the Galois closure of $F/K$ (which as observed in Section \ref{inertia} is equal to $K(\Jac(X)[4])$) is given by $F(\zeta_{2(p-1)})$, where $\zeta_{2(p-1)}$ is a primitive $2(p-1)$-th root of unity. In particular $F/K$ is Galois if and only if $f_{K/\Q_p}$ is even, and if it is odd then $[F(\zeta_{2(p-1)}):F]=2$.
	
	We are now ready to compute $\rho_\ell(\Frob_F)$, hence $\rho_\ell(\Frob_K)$.
	
	\subsection{The action of Frobenius}\label{sec:Frobenius}
	Suppose first that $n=f_{K/\Q_p}$ is even. Then $\Frob_F$ is central in $\Gal(L/K^{nr})$ (recall $L=F^{nr}$) and, since $\rho_\ell$ is irreducible, by Schur's Lemma $\rho_\ell(\Frob_F)$ is a scalar matrix. Let $\lambda \in \overline{\Q}_\ell$ be such that $\rho_\ell(\Frob_F)=\lambda \id$. Then we know $\det(\rho_\ell(\Frob_F))= |k|^g=p^{ng}$ and so $\lambda^{2g}=p^{ng}$. On the other hand since $\rho_\ell(\Frob_F)$ is a scalar matrix then its characteristic polynomial is precisely $(T-\lambda)^{2g}$, and by \cite[ Theorem 1.6]{Deligne} it has integral coefficients, so $\lambda \in \Z$ and in particular $\lambda \in \{\pm p^{n/2}\}$. Finally, since $F/K$ is Galois, we have $F=K(\Jac(X)[4])$, so $\rho_\ell(\Frob_F)$ acts trivially modulo $4$ and $\lambda \equiv 1 \pmod 4$. Hence
	$$ \lambda = \left(\lagrange{-1}{p} p\right)^{f_{K/\Q_p}/2}. $$
	
	Suppose now that $n=f_{K/\Q_p}$ is odd. In general, $\rho_\ell(\Frob_F)$ acts as the $n$-th power of the linear operator obtained for $n=1$, so we may first assume that $n=1$. Then the square of $\Frob_F$ is central in $\Gal(L/K^{nr})$, hence the minimal polynomial of $\rho_\ell(\Frob_F)$ is of the form $T^2-\mu$. As a consequence of the Weil Conjectures (see again \cite[Theorem 1.6]{Deligne}) we also have that the trace of $\rho_\ell(\Frob_F)$ is given by $p+1-|\tilde{X}_F(\F_p)|$ where $\tilde{X}_F$ is the reduction modulo $p$ of $X \times_K F$. Since for each $x \in \F_p$, $x^p-x=0$, we have precisely $p$ affine points on $\tilde{X}_F$, so $\tr(\rho_\ell(\Frob_F))=0$. Therefore the characteristic polynomial of $\rho_\ell(\Frob_F)$ has $g$ roots equal to $\sqrt{\mu}$ and $g$ roots equal to $-\sqrt{\mu}$, hence it is $(T^2-\mu)^g$, and again it has constant term equal to $p^g$ and integer coefficients, hence $\mu \in \{\pm p\}$. As in the previous case, we have $\mu \equiv 1 \pmod 4$ and so $\mu = \lagrange{-1}{p}p$. Putting all this together, the eigenvalues of $\rho_\ell(\Frob_F)$ for generic odd $f_{K/\Q_p}$ are
	$$ \pm \sqrt{\lagrange{-1}{p}p}^{f_{K/\Q_p}}, $$
	each occurring $g$ times.
	
	Now we can prove Theorem \ref{maintheorem} in the case of even $f_{K/\Q_p}$.
	
	\begin{proof}[Proof of Theorem \ref{maintheorem} for even inertia degree.]
		Since $f_{K/\Q_p}$ is even, we know $F/K$ is Galois with Galois group isomorphic to its inertia subgroup. We furthermore have 
		$$\Gal(L/K)=\Gal(F/K) \times \Gal (K^{nr}/K),$$
		since $L=F^{nr}=F K^{nr}$. If we define $\chi$ as in the statement of Theorem \ref{maintheorem} we have that $\rho_\ell(\Frob_F) = \chi(\Frob_F) \id = \chi(\Frob_K) \id$ and therefore if we let $\psi = \rho_\ell \otimes \chi^{-1}$, then $\psi$ factors through $\Gal(F/K)$ which is isomorphic to $I_G$, and as a representation of this group it is irreducible, faithful and $(p-1)$-dimensional. By Lemma \ref{I_faithful}, there exists a unique such representation.
	\end{proof}
	
	\section{The case of odd inertia degree} \label{proof}
	
	In this final section, we complete the proof of Theorem \ref{maintheorem} for the case when $f_{K/\Q_p}$ is odd, computing explicitly $\psi$.
	
	Let $\chi$ be as in the statement of Theorem \ref{maintheorem}. Then we can fix a basis of $T_\ell \Jac (X)$ such that the matrix representing $\psi(\Frob_K)=\dfrac{1}{\chi(\Frob_K)}\rho_\ell(\Frob_K)$ in this basis is diagonal with the first $g$ coefficients equal to $1$ and the last $g$ coefficients equal to $-1$. In particular $\psi(\Frob_K^2)=\id$ and so $\Frob_K^2 \in \ker (\psi)$. Therefore we have that $\ker (\psi)= \Gal(\overline K / F(\zeta_{2(p-1)}))$ and so $\psi$ factors through $G=\Gal(F(\zeta_{2(p-1)})/K)$, and it is faithful as a representation of $G$. Now $G$ is generated by $I_G$ and the element $\phi$ defined in Section \ref{statement}, with $G \cong I_G \rtimes \langle \phi \rangle$. Note that $\phi$ is the reduction of $\Frob_K$ modulo its square. The group $G$ has the following presentation:
	\[
	G = \langle \sigma, \tau, \phi | \sigma^p=\tau^{2(p-1)}=\phi^2=1, \tau \sigma \tau^{-1} = \sigma^b, \sigma \phi = \phi \sigma, \phi \tau \phi= \tau^p \rangle.
	\]
	
	In the diagram below we show the relations among the fields $K$, $K(\zeta_{2(p-1)})$, $F$, $F(\zeta_{2(p-1)})$, $L$, $\overline{K}$ and we highlight the relevant Galois groups (here $\overline \phi$ is the image of $\phi$ in $G/I_G$).
	
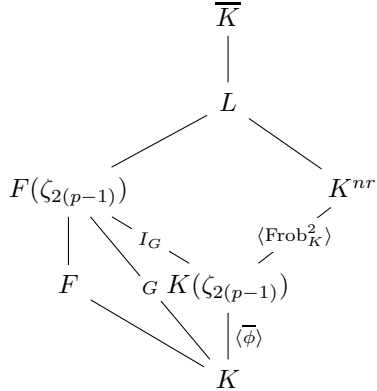
	\begin{center}
		\begin{tikzcd}[column sep=tiny]
			& \overline{K} \arrow[dash]{d} &   \\
			& L \arrow[dash]{dl} \arrow[dash]{dr}&   \\
			F(\zeta_{2(p-1)}) \arrow[dash]{ddr}[description]{G} \arrow[dash]{d} \arrow[dash]{dr}[description]{I_G} &   & K^{nr} \arrow[dash]{dl}[description]{\langle \Frob_K^2\rangle}\\
			F \arrow[dash]{dr}    & K(\zeta_{2(p-1)}) \arrow[dash]{d}{\langle \overline \phi \rangle} &   \\
			& K &   
		\end{tikzcd}
	\end{center}
	
	\begin{lemma}
		The group $G$ is isomorphic to a semidirect product
		$$ C_p \rtimes (C_{2(p-1)} \rtimes C_2).$$
	\end{lemma}
	
	\begin{proof}
		Since $\langle \phi \rangle \cong C_2$, we know that $ G \cong (C_p \rtimes C_{2(p-1)}) \rtimes C_2$. Moreover the subgroup given by wild inertia is normal, so $G$ has a normal subgroup isomorphic to $C_p$. We only need to prove that $G$ also has a subgroup isomorphic to $C_{2(p-1)}\rtimes C_2$. The field $K(\alpha_1)$ is an intermediate extension of degree $p$ over $K$, and $\Gal(F(\zeta_{2(p-1)}) /K(\alpha_1)) \cong C_{2(p-1)}\rtimes C_2$ is a subgroup of $G$.
	\end{proof}
	
	In particular $G$ is of the form $A \rtimes H$, with $A$ abelian, as in \cite[\S 8.2]{Serre}. Again we can use Proposition 25 of op. cit. to describe the irreducible representations of $G$.
	
	\subsection{The irreducible representations of the group $G$}\label{irreps_G}
	
	A set of representatives for the orbits of $H$ in the group of characters of $A$ consists, as in Section \ref{inertia}, only of the two elements $\mathbf{1}$ and $\eta$, for $\eta$ any non-trivial character. It is easy to check that, with the same notation as in Section \ref{inertia}, $H_{\mathbf{1}}=H$ and $H_{\eta} = \langle \phi, \nu\rangle \cong C_2^2$. All the representations of $H_{\mathbf{1}}$ give rise to a representation of $G$ of the same dimension. Now $H_{\mathbf{1}} \cong C_{2(p-1)} \rtimes C_2$ is itself a semidirect product of two abelian subgroups, so using Proposition 25 of \cite{Serre} we have that all its irreducible representation have dimension dividing the order of the second subgroup, that is either $2$ or $1$. However $\psi$ is irreducible of dimension $p-1$ (since $\rho_\ell$ is), so unless $p=3$ it cannot arise from such a representation. For $p=3$ we need a more direct approach, and this case is dealt with in \cite[\S 3]{1812.05651}, so we can assume $p \neq 3$.
	
	Now let us consider the representations arising from $H_\eta$. Since this group is abelian, it only has $1$-dimensional irreducible representations, namely those given by the following characters:
	$$
	\begin{array}{c|rrrr}
		\rm class& 1 & \nu & \phi & \nu \phi\cr
		\hline
		\xi_{1}&1&1&1&1\cr
		\xi_{2}&1&1&-1&-1\cr
		\xi_{3}&1&-1&1&-1\cr
		\xi_{4}&1&-1&-1&1\cr
	\end{array}
	$$
	
	The irreducible representations arising from these four representations are
	$$ \Ind_{C_p \times C_2^2}^G \xi_j \otimes \eta $$
	for $j \in \{1,\dots,4\}$ (note that the subgroup of $G$ isomorphic to $C_p \rtimes C_2^2$ is in fact a direct product). In particular these representations have dimension equal to $[G:C_p \times C_2^2]=p-1$. Following the same proof as in Lemma \ref{I_faithful} we have that only the representations arising from $\xi_3$ and $\xi_4$ are faithful, so $\psi$ is one of these two.
	
	Let $\sigma,\tau$ be the generators of $I_G$, as in the previous sections.
	
	\begin{lemma}\label{pip}
		The representations $\psi_1 = \Ind_{C_p \times C_2^2}^G \xi_3 \otimes \eta$ and $\psi_2 = \Ind_{C_p \times C_2^2}^G \xi_4 \otimes \eta$ are such that $\tr (\psi_1(\sigma \phi))=- \tr(\psi_2(\sigma \phi))= \sqrt{\lagrange{-1}{p}p}$.
	\end{lemma}
	
	\begin{proof}
		First of all, it is easy to check that $C_p \times C_2^2$ is a normal subgroup of $G$. Moreover a set of representatives for $G/(C_p \times C_2^2)$ is given by $\tau,\tau^2,\dots,\tau^{p-1}$. We have 
		$$\tr (\psi_j (\sigma \phi))= \sum_{i=1}^{p-1} (\xi_{j+2} \otimes \eta) (\tau^i\sigma \phi\tau^{-i})= \sum_{i=1}^{p-1} \xi_{j+2} (\tau^i \phi \tau^{-i}) \eta (\tau^{i} \sigma \tau^{-i}).$$
		
		By the relation $\phi \tau \phi = \tau^p$ we deduce $\tau^2 \phi = \phi \tau^2$; so if $i$ is even then $\xi_{j+2} (\tau^i \phi \tau^{-i}) = \xi_{j+2}(\phi)$, and if $i$ is odd then $\xi_{j+2} (\tau^i \phi \tau^{-i}) = \xi_{j+2}(\phi \nu) = - \xi_{j+2}(\phi)$ (recall  that $\nu = \tau^{p-1}$). On the other hand, since $\tau \sigma \tau^{-1} = \sigma^b$, then $\eta(\tau^i \sigma \tau^{-i})$ varies among all the powers of $\eta(\sigma)$, which is a primitive $p$-th root of unity, without loss of generality we can assume it is $e^{2 \pi i/p}$, seen as a complex number. Note that
		$$\lagrange{b^i}{p}=(-1)^i = \dfrac{\xi_{j+2} (\tau^i \phi \tau^{-i})}{\xi_{j+2}(\phi)},$$
		therefore
		$$\tr \psi_j (\sigma \phi) = \sum_{i=1}^{p-1} (-1)^i \xi_{j+2}(\phi) \eta(\sigma)^{b^i} = \xi_{j+2} (\phi) \sum_{a=1}^{p-1} \lagrange{a}{p} (e^{2 \pi i/p})^a =  \xi_{j+2} (\phi) \sqrt{\lagrange{-1}{p} p},$$ where the last equality follows from the Gauss summation formula.
	\end{proof}
	
	\subsection{The proof of Theorem \ref{maintheorem}}\label{sub:proof}
	
	\begin{proof}
		Let $\beta(x,y) = (x',y')$ be the change of variables described in the proof of Lemma \ref{lem:good_model}, let ``$\red$'' be the reduction map: $X (\overline{K}) \rightarrow \tilde{X}(\overline{k})$ and ``$\lift$'' be any section of it (which exists by smoothness of $\tilde{X}$). Then we can compute the action of a Galois automorphism $\gamma$ on the reduced curve $\tilde{X}(\overline{k})$ via the composition $\red \circ \beta \circ \gamma \circ \beta^{−1} \circ \lift$. In particular we will do it for $\gamma = \sigma \Frob_K$; then we have that
		$$ \tr (\rho_\ell (\sigma \Frob_K)) = |k| + 1 - A $$
		where $A$ is the number of points on the reduced curve fixed by the map $\red \circ \beta \circ \sigma \Frob_K \circ \beta^{−1} \circ \lift$ constructed above (see \cite[Theorem 1.5 and Remark 1.7]{1809.10208} and \cite[\S 6.5]{quotients}).
		
		Let $(\tilde x, \tilde y ) \in \tilde X(\overline{k})$, then since $\sigma$ is a wild inertia element we have
		\begin{eqnarray*}
			\left(\tilde{x},\tilde{y}\right) \xrightarrow{\lift} \left(x,y\right) \xrightarrow{\beta^{-1}} \left(x(\alpha_2-\alpha_1) +\alpha_1,y(\sqrt{\alpha_2-\alpha_1})^p\right) \\
			\xrightarrow{\sigma \Frob_K} \left(\sigma(\Frob_K(x)) \sigma(\alpha_2-\alpha_1) + \alpha_2, \sigma(\Frob_K(y))(\sigma(\sqrt{\alpha_2-\alpha_1}))^p\right)\\
			\xrightarrow{\beta} \left(\dfrac{\sigma(\Frob_K(x)) \sigma(\alpha_2-\alpha_1) + \alpha_2-\alpha_1}{\alpha_2-\alpha_1}, \sigma(\Frob_K(y))\dfrac{(\sigma(\sqrt{\alpha_2-\alpha_1}))^p}{(\sqrt{\alpha_2-\alpha_1})^p}\right)\\
			=\left( \sigma(\Frob_K(x)) \dfrac{\sigma(\alpha_2-\alpha_1)}{\alpha_2-\alpha_1}+1, \sigma(\Frob_K(y))\dfrac{(\sigma(\sqrt{\alpha_2-\alpha_1}))^p}{(\sqrt{\alpha_2-\alpha_1})^p}\right) \xrightarrow{\red} \left(\tilde{x}^{|k|}+1,\tilde{y}^{|k|}\right).
		\end{eqnarray*}
		
		Here we use the following facts: since $\sigma$ belongs to $I_K$, for every algebraic integer $x \in F$, $x$ and $\sigma(x)$ reduce to the same element of $\overline{k}$. If in addition $x \neq 0$ then, since $\sigma$ is wild, $\sigma(x)/x$ reduces to $1$ (again by \cite[Lemma 3.3]{1812.05651}). Finally, by definition $\Frob_K(x)$ reduces to $\tilde{x}^{|k|}$. We deduce that $A$ is equal to the number of solutions (including the point at infinity) of the following system of equations:
		\begin{align}\label{sys}\left\lbrace
			\begin{array}{cl}
				x&=x^{|k|}+1  \\
				y&=y^{|k|} \\
				y^2&=x^p-x;
			\end{array}\right.
		\end{align}
		
		As in Section \ref{sec:Frobenius} let $n = f_{K/\Q_p}$, so $|k|=p^n$. If $n=1$, then this system has $0$ affine solutions if $p \equiv 3 \pmod 4$ and $2p$ affine solutions if $p \equiv 1 \pmod 4$, so $A=1$ or $2p+1$ respectively. Therefore
		\[
		\tr (\rho_\ell (\sigma \Frob_K)) = - \lagrange{-1}{p} p,
		\]
		and so $\tr (\psi(\sigma \phi))=\dfrac{\tr (\rho_\ell(\sigma \Frob_K))}{\chi(\Frob_K)}=\dfrac{-\lagrange{-1}{p}p}{\sqrt{\lagrange{-1}{p}p}} = - \sqrt{\lagrange{-1}{p} p}$.
		
		For general odd $n$, we obtain the same system of equations independently of the curve $X$ we use, as long as it is defined over a $p$-adic field $K$ with $f_{K/\Q_p}=n$ and it satisfies the conditions of Theorem \ref{maintheorem}. Let $X/\Q_p : y^2= x^p - p$ as in Remark \ref{example}, and let $X_K$ be the base change of $X$ to the field $K$ given by the unique unramified extension of $\Q_p$ of degree $n$. The polynomial $x^p-p$ is irreducible over $K$, as any root gives a ramified extension, and $v_{K}(\Delta)=v_{\Q_p}(\Delta)=2p-1$ is coprime to $p-1$. Let $\rho_\ell'$ be the $\ell$-adic Galois representation attached to $\Jac(X)$ and let $\rho_\ell$ be the $\ell$-adic Galois representation attached to $\Jac(X_K)$; then:
		\begin{enumerate}
			\item $\rho_\ell'$ and $\rho_\ell$ have the same restriction to the inertia subgroup;
			\item $\rho_\ell(\Frob_K)$ acts as the $n$-th power of $\rho_\ell'(\Frob_{\Q_p})$.
		\end{enumerate}
		
		So $\rho_\ell(\sigma)\rho_\ell(\Frob_K) =\rho_\ell'(\sigma)\rho_\ell'(\Frob_{\Q_p})^n$. Notice that, by Section \ref{sec:Frobenius}, since $n-1$ is even, we have that $\rho_\ell'(\Frob_{\Q_p})^{n-1}$ is the scalar matrix with eigenvalue $\left(\lagrange{-1}{p}p\right)^{(n-1)/2}$. Therefore
		$$ \tr (\rho_\ell (\sigma \Frob_K)) = \left(\lagrange{-1}{p}p\right)^{(n-1)/2}  \tr (\rho_\ell'(\sigma \Frob_{\Q_p})) = -\left(\lagrange{-1}{p}p\right)^{(n+1)/2}.$$
		
		We conclude, since
		$$\tr (\psi(\sigma \phi))=\dfrac{\tr (\rho_\ell(\sigma \Frob_K))}{\chi(\Frob_K)}=\dfrac{-\left(\lagrange{-1}{p}p\right)^{(n+1)/2}}{\sqrt{\lagrange{-1}{p}p}^n} = - \sqrt{\lagrange{-1}{p} p}.$$
	\end{proof}
	
	\section{Applications}\label{applications}
	
	In this section we present a few examples and applications of Theorem \ref{maintheorem} and of the tools used throughout the paper.
	
	\begin{enumerate}
		\item By the computation made in Section \ref{sub:proof}, we find that the number $A-1$ of affine solutions of the system \eqref{sys} is
		\[
		A - 1 = |k| - \tr(\rho_\ell(\sigma \Frob_K)) = p^n + \left(\lagrange{-1}{p} p\right)^{(n+1)/2}.
		\]
		
		\item We can express the representation $\psi$ given in Theorem \ref{maintheorem} in terms of the characters introduced in Section \ref{irreps_G}. With the same notation, we have that
		$$\psi = \psi_2 = \Ind_{C_p \times C_2^2}^G \xi_4 \otimes \eta.$$
		
		A few examples are the following:

		\begin{enumerate}
			\item For $p=5$, the group $I_G$ is isomorphic to $C_5 \rtimes C_8$ in \cite{groupnames}. The restriction of $\rho_\ell$ to inertia is given by $\rho_{10}$. For odd $f_{K/\Q_p}$, we have $G \cong C_2^2\cdot F_5$ in \cite{groupnames}, with $\psi=\rho_{13}$ (here the class denoted $10A$ is the one generated by $\sigma \phi$).
			\item For $p=7$, $I_G \cong C_7 \rtimes C_{12}$ and the corresponding representation is $\rho_{14}$. For odd $f_{K/\Q_p}$, we have $G \cong \Dic_7 \rtimes C_6$ with $\psi = \rho_{18}$ (here the class $14A$ is the one generated by $\sigma \phi$).
		\end{enumerate}
		
		\item In Section \ref{irreps_G}, we assumed that $p \neq 3$. For $p=3$, Theorem \ref{maintheorem} still holds and a complete proof is presented in \cite[\S 3]{1812.05651}. Both the results of \cite{1812.05651} and of this paper have been implemented in MAGMA and are available on \cite{github}.
		
		\item Under the assumptions of Theorem \ref{maintheorem}, it is possible to compute the exponent $N$ of the conductor of the curve $X$, by using \cite[\S 11, Theorem 11.3]{M2D2}. In fact, by the description of the cluster picture of $X$ made in Remark \ref{cluster}, there is only one Galois orbit of the set of roots of $f$, corresponding to the only proper cluster. Let $\alpha_1$ be any root of $f$. Then
		\[
		N = v_K(\Delta_{K(\alpha_1)/K}) - [K(\alpha_1):K] + f_{K(\alpha_1)/K} + 2g,
		\]
		where $\Delta_{K(\alpha_1)/K}$ is the discriminant of the field extension $K(\alpha_1)/K$. Now, since $K(\alpha_1)/K$ is totally ramified of degree $p$ and since $2g=p-1$, we obtain simply 
		$$N = v_K(\Delta_{K(\alpha_1)/K}).$$
		
		For example if $f(x)=x^p-p$, as in Remark \ref{example}, then $N=2p-1$.

	\end{enumerate}
	
	\section{Acknowledgements}
	The author wishes to thank her advisor Tim Dokchitser for the supervision throughout the writing of this paper, Davide Lombardo and the anonymous referee for their careful reading and helpful comments, Pip Goodman for the proof of Lemma \ref{pip}, and Nuno Freitas and John Voight for suggesting to compute the conductor.
	
	The research was supported by EPSRC.

\end{document}